\documentclass[12pt]{amsart}
\usepackage{ amsmath,amssymb,amsthm,color, euscript, enumerate}
\usepackage{dsfont}
\usepackage{tikz}

\newcommand{\1}{ \mathds{1}}

\newcommand{\Spec}{\mathrm{Spec}\ }

\newcommand{\maru}[1]{{\ooalign{\hfil#1\/\hfil\crcr
\raise.167ex\hbox{\mathhexbox20D}}}}

\newcommand{\ruby}[2]{%
 \leavevmode
 \setbox0=\hbox{#1}%
 \setbox1=\hbox{\tiny #2}%
 \ifdim\wd0>\wd1 \dimen0=\wd0 \end{lemma}se \dimen0=\wd1 \fi
 \hbox{%
   \kanjiskip=0pt plus 2fil
   \xkanjiskip=0pt plus 2fil
   \vbox{%
     \hbox to \dimen0{%
       \tiny \hfil#2\hfil}%
     \nointerlineskip
     \hbox to \dimen0{\mathstrut\hfil#1\hfil}}}}

\newcommand{\Z}{\mathbb{Z}}
\newcommand{\C}{\mathbb{C}}
\newcommand{\R}{\mathbb{R}}

\newcommand{\Aut}{\mathrm{Aut}\,}

\makeatletter \@addtoreset{equation}{section}

\theoremstyle{plain}
\newtheorem{theorem}{Theorem}[section]
\newtheorem{proposition}[theorem]{Proposition}
\newtheorem{lemma}[theorem]{Lemma}

\theoremstyle{definition}
\newtheorem{definition}[theorem]{Definition}

\theoremstyle{remark}
\newtheorem{remark}[theorem]{Remark}
\numberwithin{equation}{section}

\title[Holomorphic vertex operator algebra of type $A_{6,7}$]{A holomorphic vertex operator algebra of central charge $24$ whose weight one Lie algebra has type $A_{6,7}$}
\author{Ching Hung Lam} %
  \address[C. H. Lam] {Institute of Mathematics, Academia Sinica, Taipei 10617, Taiwan and National Center for Theoretical Sciences of  Taiwan.}
  \email{chlam@math.sinica.edu.tw}
\author[H. Shimakura]{Hiroki Shimakura}%
\address[H. Shimakura]{Graduate School of Information Sciences,
Tohoku University,
Sendai 980-8579, Japan }%
\email {shimakura@m.tohoku.ac.jp}%
\date{}
\thanks{C.\,H. Lam was partially supported by MoST grant 104-2115-M-001-004-MY3 of Taiwan}
\thanks{H.\ Shimakura was partially supported by JSPS KAKENHI Grant Numbers 26800001.}
\thanks{C.\,H. Lam and H.\ Shimakura were partially supported by JSPS Program for Advancing Strategic International Networks to Accelerate the Circulation of Talented Researchers ``Development of Concentrated Mathematical Center Linking to Wisdom of the Next Generation".}

\newcommand{\sfr}[2]{\leavevmode\kern-.1em
  \raise.5ex\hbox{\the\scriptfont0 #1}\kern-.1em
  /\kern-.15em\lower.25ex\hbox{\the\scriptfont0 #2}}

\begin{document}

\begin{abstract}
In this article, we describe a construction of  a holomorphic vertex operator algebras of central charge $24$ whose weight one Lie algebra has type $A_{6,7}$. 
\end{abstract}
\maketitle


\section{Introduction}
In 1993, Schellekens \cite{Sc93} obtained a list of $71$ possible Lie algebra structures for the weight one subspaces of holomorphic vertex operator algebras (VOAs) of central charge $24$. However, only $39$ of the $71$ cases in his list had been constructed explicitly at that time. 
In the recent years, many new holomorphic VOAs of central charge $24$ have been constructed. In \cite{Lam,LS}, $17$ holomorphic VOAs were constructed using the theory of framed VOAs. In addition, three holomorphic VOAs were constructed in \cite{Mi3,SS}
using $\Z_3$-orbifold constructions associated to  lattice VOAs.
Recently, van Ekeren, M\"oller and Scheithauer \cite{EMS} have established the general $\Z_n$-orbifold construction for elements of  arbitrary orders. In particular,  constructions of five holomorphic VOAs were discussed. In \cite{LS3}, five other holomorphic VOAs were constructed using $\Z_2$-orbifold constructions associated to inner automorphisms. Based on these results, there are two remaining cases in Schellekens' list which have not been discussed yet. The corresponding Lie algebras have type $A_{6,7}$ and $F_{4,6}A_{2,2}$.

In this article, we will describe a construction of a holomorphic VOA of central charge $24$ whose weight one Lie algebra has type $A_{6,7}$.  
Since the level of  $A_{6,7}$ is $7$, it is natural to hope that a $\Z_7$-orbifold construction will bear fruit.  It is indeed correct and we will show that the desired VOA can be constructed by applying the $\Z_7$-orbifold construction to the Leech lattice VOA $V_\Lambda$ and an order $7$ automorphism of $V_\Lambda$; however, the choice of the automorphism is somewhat tricky. It is the product of (a lift of) an order $7$ isometry of $\Lambda$ and an order $7$ inner automorphism of $V_\Lambda$.
Combining the explicit construction of the twisted $V_\Lambda$-modules for an isometry of $\Lambda$ in \cite{DL} and the modification by Li's $\Delta$-operator in \cite{Li}, we obtain the irreducible twisted $V_\Lambda$-modules for the product of these two order $7$ automorphisms.
By using some explicit information, we will then show that  the weight one subspace of the resulting orbifold VOA has dimension $48$ and it is a Lie algebra of type $A_6$.

Let us explain the two automorphisms in more detail.
It is known (see \cite{Atlas}) that the Leech lattice has  exactly two conjugacy classes of isometries of order $7$. One class acts fixed-point-freely on $\Lambda$ and the other class acts on $\Lambda$ with trace $3$. By the explicit construction given in \cite{DL}, for an element in the former (resp. latter) class, 
the irreducible twisted $V_\Lambda$-module has lowest $L(0)$-weight $8/7$ (resp. $6/7$) and the resulting orbifold VOA will have a trivial (resp. $24$-dimensional) weight one subspace.
In either case, the weight one subspace is too small. 
The main trick for our construction is to modify the irreducible $\tau$-twisted $V_\Lambda$-module associated to an order $7$ isometry $\tau$ with trace $3$ by Li's $\Delta$-operator (cf. Proposition \ref{Prop:twist}) so that the resulting irreducible twisted $V_\Lambda$-module has lowest $L(0)$-weight $1$. It is equivalent to find a vector $f\in (1/7)\Lambda$ such that $f$ is fixed by $\tau$ and $(f|f)=2/7$ (see \eqref{def_f}). 
Then, if we set $g=\sigma_f\cdot\tau$, where $\sigma_f$ is the inner automorphism associated to $f$, then $g$ is the desired order $7$ automorphism of $V_\Lambda$ and the modified module is an irreducible $g$-twisted $V_\Lambda$-module. 

Our argument is, in fact, very similar to that in \cite[Section 10]{LS3}, in which a construction of a holomorphic VOA of central charge $24$ such that the corresponding weight one Lie algebra has type $A_{4,5}^2$ is discussed.   

The organization of this article is as follows. In Section 2, we recall some basic facts about VOAs and review Li's $\Delta$-operator. In Section 3, we review some basic facts about the isometry  group of the Leech lattice $\Lambda$ and describe an order $7$ isometry of $\Lambda$.
In addition, we find a suitable vector $f\in\Lambda$ and prove some lemmas about lattices.
In Section 4,  we discuss  a construction of a holomorphic vertex operator algebra of central charge $24$ such that the weight one Lie algebra has type $A_{6,7}$.  

\section{Preliminary}
In this section, we will review some fundamental results about VOAs.
We adopt the terminology and notation used in \cite{LS3}.

Let $V$ be a VOA of CFT-type.
Then, the weight one space $V_1$ has a Lie algebra structure via the $0$-th mode, which we often call the {\it weight one Lie algebra} of $V$.
 Moreover, the $n$-th modes
$v_{(n)}$, $v\in V_1$, $n\in\Z$, define  an affine representation of the Lie algebra $V_1$ on $V$.
For a simple Lie subalgebra $\mathfrak{a}$ of $V_1$, the {\it level} of $\mathfrak{a}$ is defined to be the scalar by which the canonical central element acts on $V$ as the affine representation.
When the type of the root system of $\mathfrak{a}$ is $X_n$ and the level of $\mathfrak{a}$ is $k$, we denote the type of $\mathfrak{a}$ by $X_{n,k}$.

\begin{proposition} [{\cite[(1.1), Theorem 3 and Proposition 4.1]{DMb}}]\label{Prop:V1} Let $V$ be a strongly regular, holomorphic VOA of central charge $24$.
If the Lie algebra $V_1$ is neither $\{0\}$ nor abelian, then $V_1$ is semisimple, and the conformal vectors of $V$ and the subVOA generated by $V_1$ are the same.
In addition, for any simple ideal of $V_1$ at level $k$, the identity$$\frac{h^\vee}{k}=\frac{\dim V_1-24}{24}$$
holds, where $h^\vee$ is the dual Coxeter number.
\end{proposition}

Let $V$ be a VOA of CFT-type.
Let $h\in V_1$ such that $h_{(0)}$ acts semisimply on $V$. We also assume that  
$ \Spec h_{(0)}\subset (1/T) \Z$ for some positive integer $T$, 
where $\Spec h_{(0)}$ denotes 
the set of spectra of $h_{(0)}$ on $V$. 
Let $\sigma_h=\exp(-2\pi\sqrt{-1}h_{(0)})$ be the (inner) automorphism of $V$ associated to $h$.
Then by the assumption on $\Spec h_{(0)}$, we have $\sigma_h^T=1$.
Let $\Delta(h,z)$ be Li's $\Delta$-operator defined in \cite{Li}, i.e.,
\[
\Delta(h, z) = z^{h_{(0)}} \exp\left( \sum_{n=1}^\infty \frac{h_{(n)}}{-n} (-z)^{-n}\right).
\]
 
\begin{proposition}[{\cite[Proposition 5.4]{Li}}]\label{Prop:twist}
Let $\sigma$ be an automorphism of $V$ of finite order and
let $h\in V_1$ be as above such that $\sigma(h) = h$.
Let $(M, Y_M)$ be a $\sigma$-twisted $V$-module and
define $(M^{(h)}, Y_{M^{(h)}}(\cdot, z)) $ as follows:
\[
\begin{split}
& M^{(h)} =M \quad \text{ as a vector space;}\\
& Y_{M^{(h)}} (a, z) = Y_M(\Delta(h, z)a, z)\quad \text{ for any } a\in V.
\end{split}
\]
Then $(M^{(h)}, Y_{M^{(h)}}(\cdot, z))$ is a
$\sigma_h\sigma$-twisted $V$-module.
Furthermore, if $M$ is irreducible, then so is $M^{(h)}$.
\end{proposition}

Assume that $V$ is self-dual.
Then there exists a symmetric invariant bilinear form $\langle\cdot|\cdot\rangle$ on $V$, which is unique up to scalar (\cite{Li3}).
We normalize it so that  $\langle\1|\1\rangle=-1$.
Then for $a,b\in V_1$, we have $\langle a|b\rangle\1=a_{(1)}b$.

For a $\sigma$-twisted $V$-module $M$ and $a\in V$, we denote by $a_{(i)}^{(h)}$ the operator which corresponds to the coefficient of $z^{-i-1}$ in $Y_{M^{(h)}}(a,z)$, i.e.,
$$Y_{M^{(h)}}(a,z)=\sum_{i\in\Z}a_{(i)}^{(h)}z^{-i-1}\quad \text{for}\quad a\in V.$$
The $0$-th mode of an element $x\in V_1$ on $M^{(h)}$ is given by
\begin{equation}
x^{(h)}_{(0)}=x_{(0)}+\langle h|x\rangle {\rm id}.\label{Eq:V1h}
\end{equation}
Let us denote by $L^{(h)}(n)$ the $(n+1)$-th mode of the conformal vector $\omega\in V$ on $M^{(h)}$.
Then the $L(0)$-weights on $M^{(h)}$ are given by
\begin{equation}
L^{(h)}(0)=L(0)+h_{(0)}+\frac{\langle h|h\rangle}{2}{\rm id}.\label{Eq:Lh}
\end{equation}

\section{Leech lattice and its isometry group}\label{sec:4}
We review some notations and certain basic properties
of the Leech lattice $\Lambda$ and its isometry group $O(\Lambda)$, which is also known as $Co_0$, a perfect group of order $2^{22}\cdot 3^9\cdot 5^4\cdot 7^2\cdot 11\cdot 13\cdot 23$.

\subsection{Hexacode balance and the Leech lattice}

Let $\Omega=\{1,2,3,\dots, 24\}$ be a set of $24$ elements and let
$\mathcal{G}$ be the (extended) Golay code of length $24$ indexed by $\Omega$. A subset
$S\subset \Omega$ is called a $\mathcal{G}$-set if
$S=\mathrm{supp}\,\alpha$ for some codeword $\alpha\in \mathcal{G}$. We
will identify a $\mathcal{G}$-set with the corresponding codeword
in $\mathcal{G}$. 
A \textit{sextet} is a
partition of $\Omega$ into six $4$-element sets of which the union
of any two forms a $\mathcal{G}$-set. 

For explicit calculations,  we use the notion of \textit{hexacode balance} to denote the codewords of the Golay code and the vectors in the Leech lattice (cf. \cite[Chapter 11]{CS} and \cite[Chapter 5]{G12}). First we arrange  the set $\Omega$ into
a $4\times 6$ array such that the six columns form a sextet.
For each codeword in $\mathcal{G}$, $0$ and $1$ are marked by a
blanked and non-blanked space, respectively, at the corresponding
positions in the array. For example, $(1^8 0^{16})$ is denoted by
the array
\begin{equation}
\begin{array}{|cc|cc|cc|}
\hline *& *&\ &\ &\ &\ \\
* & *&\ &\ &\ &\ \\
* & *&\ &\ &\ &\ \\
* & *&\ &\ &\ &\ \\ \hline
\end{array}\label{Eq:1^80^16}
\end{equation}

Let $(\cdot|\cdot)$ be the inner product of $\R^{24}$ and let $e_1,e_2,\dots,e_{24}\in\R^{24}$ be an orthogonal basis of squared norm $2$, that is, $(e_i|e_j)=2\delta_{i,j}$ for $1\le i,j\le 24$.
Denote $e_{X}:=\sum_{i\in X}e_{i}$ for $X\in \mathcal{G}$.
The following is a standard construction of the Leech lattice.
\begin{definition}[\cite{CS}]\label{gen}
 The \textit{(standard) Leech lattice} $\Lambda $ is a lattice of rank 24 generated by
the vectors:
\begin{eqnarray*}
&&\frac{1}{2}e_{X}\, ,\quad \text{where } X\text{ is a generator of the Golay code }%
\mathcal{G}; \\
&&\frac{1}{4}e_{\Omega }-e_{1}\,\,; \\
&&e_{i}\pm e_{j}\,,\quad \text{ }i,j \in \Omega.
\end{eqnarray*}
\end{definition}

\begin{remark}
By arranging the set $\Omega$ into a $4\times 6$ array, every vector in the Leech lattice $\Lambda$ can be written as the form
\[
X=\frac{1}{\sqrt{8}}\left[ X_1 X_2 X_3 X_4 X_5 X_6\right],\quad \text{
juxtaposition of column vectors}.
\]
For example,
\[
\frac{1}{\sqrt{8}}\,
\begin{array}{|rr|rr|rr|}
 \hline 2  & 2& 0  & 0 & 0 & 0 \\
 2 &  2 &  0&  0& 0& 0\\
 2 &  2& 0& 0&  0 & 0 \\
 2 &  2& 0&  0& 0 & 0\\ \hline
\end{array}
\]
denotes the vector $\frac1{2}\, e_{(1^80^{16})}$, where $(1^80^{16})$ is the
codeword corresponding to \eqref{Eq:1^80^16}.
\end{remark}

\subsection{An isometry of the Leech lattice of order $7$}\label{SS:ord7}
In this subsection, we study a certain automorphism $\tau$ of $\Lambda$ of order $7$.

By the character table (cf. \cite[Page 184]{Atlas}), there exist exactly two conjugacy classes of elements of order $7$ in $O(\Lambda)$.  
Let $\tau$ be an isometry of $\Lambda$ of order $7$ such that the trace of $\tau$ on $\Lambda$ is $3$. Such elements form a unique conjugacy class in $O(\Lambda)$ and the set of fixed-points of $\tau$ in $\Lambda$ is a sublattice of rank $6$.  
For the simplicity of calculation, we fix $\tau$  such that  (see \cite[Figure 11.21]{CS})
\begin{center}
\begin{tikzpicture}
\draw[step=1cm] (0,0) grid (3,2);
\fill(0.25,0.25) circle(1.5pt);
\fill(0.75,0.25) circle(1.5pt);
\fill(1.25,0.25) circle(1.5pt);
\fill(1.75,0.25) circle(1.5pt);
\fill(2.25,0.25) circle(1.5pt);
\fill(2.75,0.25) circle(1.5pt);
\fill(0.25,0.75) circle(1.5pt);
\fill(0.75,0.75) circle(1.5pt);
\fill(1.25,0.75) circle(1.5pt);
\fill(1.75,0.75) circle(1.5pt);
\fill(2.25,0.75) circle(1.5pt);
\fill(2.75,0.75) circle(1.5pt);
\fill(0.25,1.25) circle(1.5pt);
\fill(0.75,1.25) circle(1.5pt);
\fill(1.25,1.25) circle(1.5pt);
\fill(1.75,1.25) circle(1.5pt);
\fill(2.25,1.25) circle(1.5pt);
\fill(2.75,1.25) circle(1.5pt);
\fill(0.25,1.75) circle(1.5pt);
\fill(0.75,1.75) circle(1.5pt);
\fill(1.25,1.75) circle(1.5pt);
\fill(1.75,1.75) circle(1.5pt);
\fill(2.25,1.75) circle(1.5pt);
\fill(2.75,1.75) circle(1.5pt);
\draw[thick] (0.25,1.25) -- (0.75, 1.25);
\draw[thick] (0.75,0.75) -- (0.75,1.25);
\draw[thick] (0.75,0.75) .. controls (0.95,1.25)  .. (0.75,1.75);
\draw[thick] (0.25,1.25) -- (0.75,0.25);
\draw[thick] (0.75,0.25) -- (0.25, 0.75);
\draw[thick][->] (0.25, 0.75) --  (0.25,0.27);
\draw[thick] (1.25,1.25) -- (1.75, 1.25);
\draw[thick] (1.75,0.75) -- (1.75,1.25);
\draw[thick] (1.75,0.75) .. controls (1.95,1.25)  .. (1.75,1.75);
\draw[thick] (1.25,1.25) -- (1.75,0.25);
\draw[thick] (1.75,0.25) -- (1.25, 0.75);
\draw[thick][->] (1.25, 0.75) --  (1.25,0.27);
\draw[thick] (2.25,1.25) -- (2.75, 1.25);
\draw[thick] (2.75,0.75) -- (2.75,1.25);
\draw[thick] (2.75,0.75) .. controls (2.95,1.25)  .. (2.75,1.75);
\draw[thick] (2.25,1.25) -- (2.75,0.25);
\draw[thick] (2.75,0.25) -- (2.25, 0.75);
\draw[thick][->] (2.25, 0.75) --  (2.25,0.27); 
\draw (-0.5,1) node{$\tau=$}; 
\draw (3.1,0.9) node{.};
\end{tikzpicture}
\end{center}

Let $\mathfrak{h}= \C\otimes_\Z \Lambda$.
We extend the form $(\cdot|\cdot)$  to $\mathfrak{h}$ $\C$-bilinearly.
We also extend the isometry $\tau$  to $\mathfrak{h}$ $\C$-linearly.
Let $\mathfrak{h}_{(0)}$ be the fixed-point subspace of $\tau$ in $\mathfrak{h}$ and let $P_0$ be the orthogonal projection from $\mathfrak{h}$ to $\mathfrak{h}_{(0)}$.

\begin{lemma}\label{Eq:M}
Let $ M=((1-P_0)\mathfrak{h}) \cap \Lambda$ and $r\in\{1,2,\dots,6\}$.  
Then $M=(1-\tau^r)\Lambda$.
\end{lemma}

\begin{proof} 

We note that $(1-P_0)\mathfrak{h} = \mathfrak{h}_{(0)}^\perp = \{y\in \mathfrak{h} \mid ( y|x) =0 \text{ for all } x\in \mathfrak{h}_{(0)}\}$.  Hence, $ M= \{\alpha\in \Lambda \mid ( \alpha|x) =0 \text{ for all } x\in \mathfrak{h}_{(0)}\}$.
It is clear that $(1-\tau^r)\Lambda \subset M$ since $(1-\tau^r)\Lambda$ is orthogonal to $\mathfrak{h}_{(0)}$.

By the definition of $\tau$, we also know that the fixed-point subspace $\mathfrak{h}_{(0)}$ of $\tau^r$ is given by 
\[
\left\{ \left. 
\begin{array}{|cc|cc|cc|}
 \hline  a_0 & a& b_0 & b& c_0  & c \\
 a& a& b &  b &  c&   c\\
 a& a & b &  b& c&  c \\
 a& a & b &  b& c   & c\\ \hline
\end{array}
\ \right | a_0, b_0,c_0, a,b,c \in \C  \right\},
\]
and $M$  is given by 
\[
\left\{ \left.
\begin{array}{|cc|cc|cc|}
 \hline  0 & a_1& 0 & b_1& 0  & c_1 \\
 a_2& a_3& b_2 &  b_3 &  c_2&   c_3\\
 a_4& a_5 & b_4 &  b_5& c_4&  c_5 \\
 a_6& a_7 & b_6 &  b_7& c_6   & c_7\\ \hline
\end{array}
\in \Lambda\,  \right |\,   \sum_{i=1}^7 a_i= \sum_{i=1}^7 b_i=\sum_{i=1}^7 c_i=0 \right\}. 
\] 
Note that the indexes in $\Omega$ corresponding to $a_0,b_0,c_0$ are $1$, $9$, $17$, respectively.

Let $\mathcal{G}_0=\{(c_1,\dots,c_{24})\in\mathcal{G}\mid c_1=c_9=c_{17}=0\}.$
Then the dimension of $\mathcal{G}_0$ is $9$.
Clearly, $(1-\tau^r)\mathcal{G}\subset\mathcal{G}_0$.
By using explicit generators of $\mathcal{G}$ (see for example \cite[(5.35)]{G12}), it is straightforward to show that $(1-\tau^r)\mathcal{G}$ is also a $9$-dimensional subcode of $\mathcal{G}_0 $ and hence we have 
\begin{equation}
(1-\tau^r)\mathcal{G}=\mathcal{G}_0.\label{Eq:G0}
\end{equation}

For $Y\subset\Omega$, let $\varepsilon_Y$ denote the involution in $O(\Lambda)$ that acts by $-1$ on the coordinates corresponding to the elements in $Y$.  Let $E^1$, $E^2$ and $E^3$ be the $\mathcal{G}$-sets corresponding to the codewords
$(1^8 ,0^{8},0^8)$, $(0^8, 1^8, 0^8)$, and $(0^{8},0^8, 1^8)$, respectively.
 
By the generators of $\Lambda$ given in Definition \ref{gen}, $M$ is generated by 
\[
\begin{split}
\varepsilon_Y\frac{1}{2}e_{X},\quad & \text{where } X\text{ is a generator of the code }%
\mathcal{G}_0\\ & \text{and } Y\subset\Omega\
 \text{with }  |X\cap Y\cap E^i|=|X\cap E^i|/2 \text{ for } i=1,2,3; \\
e_{i}-e_{j}\,, \quad & \text{where }i,j\in E^1\setminus\{1\},\ i,j\in E^2\setminus\{9\} \text{ or }\ i,j\in E^3\setminus \{17\}.
\end{split}
\]
By \eqref{Eq:G0}, one can easily see that $M\subset (1-\tau^r)\Lambda$.
Thus $M=(1-\tau^r)\Lambda$ as desired. 
\end{proof}

\subsection{Calculations in $P_0(\Lambda)$}
In this subsection, we prove some lemmas about $P_0(\Lambda)$, which will be used later.
Recall that $P_0$ is the orthogonal projection from $\mathfrak{h}$ to $\mathfrak{h}_{(0)}$.

\begin{lemma}\label{P0}
The lattice $P_0(\Lambda)$ is spanned over $\Z$ by the following vectors:
\begin{align*}
&\frac{1}{7\sqrt{8}}\,
\begin{array}{|rr|rr|rr|}
 \hline  0 & 4& 0 & 4& 0  & 0 \\
 4& 4& 4 &  4 &  0&   0\\
 4& 4 & 4 &  4& 0&  0 \\
 4& 4 & 4 &  4& 0   & 0\\ \hline
\end{array},  
&& \frac{1}{7\sqrt{8}}\,
\begin{array}{|rr|rr|rr|}
 \hline  0  & 0& 0 & 4& 0 & 4 \\
 0  & 0& 4& 4& 4 &  4 \\
 0  & 0& 4& 4 & 4 &  4 \\
 0  & 0& 4& 4 & 4 &  4\\ \hline
\end{array}, \\
&\frac{1}{7\sqrt{8}}\,
\begin{array}{|rr|rr|rr|}
 \hline  0 & 4& 0 & -4& 0  & 0 \\
 4& 4& -4 &  -4 &  0&   0\\
 4& 4 & -4 &  -4& 0&  0 \\
 4& 4 & -4 &  -4& 0   & 0\\ \hline
\end{array},  &
&\frac{1}{7\sqrt{8}}\,
\begin{array}{|rr|rr|rr|}
 \hline -14  & 2& 0  & 0 & 0 & 0 \\
 2 &  2 &  0&  0& 0& 0\\
 2 &  2& 0& 0&  0 & 0 \\
 2 &  2& 0&  0& 0 & 0\\ \hline
\end{array}, \\
& \frac{1}{7\sqrt{8}}\,
\begin{array}{|rr|rr|rr|}
 \hline  0 & 0& -14  & 2& 0  & 0 \\
 0& 0& 2 &  2 &  0&   0\\
 0&  0 & 2 &  2& 0&  0 \\
 0& 0 & 2 &  2& 0   & 0\\ \hline
\end{array}, 
&& 
\frac{1}{7\sqrt{8}}\,
\begin{array}{|rr|rr|rr|}
 \hline  -7 & 1& -7  & 1& -7  & -3 \\
 1& 1& 1 &  1 &  -3&  -3\\
 1&  1 & 1 &  1& -3&  -3 \\
 1& 1 & 1 &  1& -3   & -3\\ \hline
\end{array}. \\
\end{align*}
\end{lemma}

\begin{proof}
Note that $P_0=\frac{1}7\sum_{i=0}^6 \tau^i$. The result now follows easily by a direct calculation using Definition \ref{gen} and a basis for the Golay code (cf. \cite[(5.35)]{G12}). 
\end{proof}

We now set 
\begin{equation} \label{def_f}
f= \frac{1}{7\sqrt{8}}\,
\begin{array}{|rr|rr|rr|}
 \hline  5 & 1& 1  & 1& 3  & 3 \\
 1& 1& 1 &  1 &  3&  3\\
 1&  1 & 1 &  1& 3&  3 \\
 1& 1 & 1 &  1& 3   & 3\\ \hline
\end{array}\in\mathfrak{h}_{(0)}.
\end{equation}
Then $(f|f)=2/7$ and $f\notin P_0(\Lambda)$.  

\begin{lemma}\label{Lem:S}For $r\in\{\pm1,\pm2,\pm3\}$, we set
\[
\mathcal{S}^{r} = \left\{ a+ rf \ \left|\ a \in P_0(\Lambda)\text{ and }  |a+ rf|^2 =\frac{2}7\right. \right\}.
\]
Then
\begin{align*}
&\mathcal{S}^1=\{\beta_0, \beta_1,\beta_2,\beta_3,\beta_4,\beta_5,\beta_6\}, && \mathcal{S}^{-1} = -\mathcal{S}^1,\\
&\mathcal{S}^2=\{\beta_0+\beta_1, \beta_1+\beta_2,\beta_2+\beta_3,\beta_3+\beta_4,\beta_4 +\beta_5,\beta_5+\beta_6,\beta_6+\beta_0\}, && \mathcal{S}^{-2} = -\mathcal{S}^2,\\
&\mathcal{S}^3=\{\beta_0+\beta_1+\beta_2,\beta_1+\beta_2+\beta_3,\beta_2+\beta_3+\beta_4,\dots, \beta_6+\beta_0+\beta_1\}, && \mathcal{S}^{-3} = -\mathcal{S}^3,
\end{align*}
where $\beta_1=f$,
\begin{align*}
\beta_2=\frac{1}{7\sqrt{8}}\,
\begin{array}{|rr|rr|rr|}
 \hline  -2& -2& -6  & -2& -4  & 0 \\
 -2& -2& -2 &  -2 &  0&  0\\
 -2&  -2 & -2 &  -2& 0&  0 \\
 -2& -2 & -2 &  -2& 0   & 0\\ \hline
\end{array}, &
&\beta_3= \frac{1}{7\sqrt{8}}\,
\begin{array}{|rr|rr|rr|}
 \hline  -2& 2& 8  & 0& -4  & 0 \\
 2& 2& 0 &  0 &  0&  0\\
 2& 2 & 0 &  0& 0&  0 \\
 2& 2 & 0 &  0& 0   & 0\\ \hline
\end{array},  \\
\beta_4=\frac{1}{7\sqrt{8}}\,
\begin{array}{|rr|rr|rr|}
 \hline  5 & -3& 1  & 1& 3  & -1 \\
 -3& -3& 1 &  1 &  -1&  -1\\
 -3&  -3 & 1 &  1& -1&  -1 \\
 -3& -3 & 1 &  1& -1   & -1\\ \hline
\end{array}, &
&\beta_5=\frac{1}{7\sqrt{8}}\,
\begin{array}{|rr|rr|rr|}
 \hline  -2& 2& -6  & 2& -4  & 0 \\
 2& 2& 2 &  2 &  0&  0\\
 2&  2 & 2 &  2& 0&  0 \\
 2& 2 &  2 &  2& 0   & 0\\ \hline
\end{array},\\
\beta_6= \frac{1}{7\sqrt{8}}\,
\begin{array}{|rr|rr|rr|}
 \hline  5 & 1& 1  &  -3  & 3& -1 \\
 1& 1 &  -3&  -3& -1 &  -1\\
 1&  1 & -3&  -3 & -1 &  -1\\
 1& 1 & -3   & -3& -1 &  -1\\ \hline
\end{array}, &
&
\beta_0=\frac{1}{7\sqrt{8}}\,
\begin{array}{|rr|rr|rr|}
 \hline  -9 & -1& 1  & 1& 3  & -1 \\
 -1& -1& 1 &  1 &  -1& -1\\
 -1&  -1 & 1 & 1& -1&  -1 \\
 -1& -1 & 1 &  1& -1   & -1\\ \hline
\end{array}. 
&& 
\end{align*}
\end{lemma}

\begin{proof}We discuss the case $r=1$ only; the other cases can be proved by a similar argument.

Since $|f|^2=2/7$, it follows that $|a+f|^2= 2/7$ if and only if $(a|f)= -\frac{1}2 |a|^2$. By the Schwarz inequality, we also have
\[
|( a| f)| \leq \sqrt{\frac{2}7} |a|.
\]
Hence, $|a+f|^2= 2/7$ implies $|a|^2\leq 8/7$. Now using Lemma \ref{P0}, it is straightforward to determine $\mathcal{S}^1$. 
\end{proof}

\begin{remark}\label{Rem:A6}
 We also note that $\beta_i$'s satisfy the relation
\[
(\beta_i|\beta_j)=
\begin{cases}
\frac{2}7 &\text{ if } i=j,\\
-\frac{1}7  &\text{ if } i-j\equiv\pm1\pmod7,\\
0 &\text{ otherwise}.
\end{cases}
\]
\end{remark}

\section{Holomorphic VOA of central charge $24$ with Lie algebra $A_{6,7}$}
In this section, we describe how to construct a holomorphic VOA whose weight one Lie algebra has type $A_{6,7}$.

\subsection{Irreducible twisted $V_\Lambda$-modules}

Let $\Lambda$ be the Leech lattice and $\tau$ the isometry of $\Lambda$ of order $7$ given in \S \ref{SS:ord7}. 
Let $V_\Lambda$ be the lattice VOA associated to $\Lambda$.
Note that the restriction of the invariant form $\langle\cdot|\cdot\rangle$ of $V_\Lambda$ to $(V_\Lambda)_1$ coincides with the form $(\cdot|\cdot)$ on $\mathfrak{h}=\C\otimes_\Z\Lambda$ via the canonical injective map 
\begin{equation}
\mathfrak{h}\to (V_\Lambda)_1,\quad h\mapsto h{(-1)}\cdot\1.\label{Eq:inj}
\end{equation}
Since the order of $\tau$ is odd, there exists a lift of $\tau$ in $V_\Lambda$ of order $7$; we also denote it by the same symbol $\tau$.
Let $V_\Lambda[\tau^r]$, ($r=\pm 1,\pm 2, \pm 3$), be the unique irreducible $\tau^r$-twisted $V_\Lambda$-module (\cite{DLM2}).
Such a module was constructed in \cite{DL} explicitly (see \cite[Section 2.2]{SS} for a review); as a vector space,
$$V_\Lambda[\tau^r]\cong M(1)[\tau^r]\otimes\C[P_0(\Lambda)]\otimes T_r,$$
where $M(1)[\tau^r]$ is the ``$\tau^r$-twisted" free bosonic space and $T_r$ is the unique irreducible module of $\hat{M}/\widehat{(1-\tau^r)\Lambda}$ satisfying certain conditions (see \cite[Propositions 6.1 and 6.2]{Le} and \cite[Remark 4.2]{DL} for details).
It follows from $(1-\tau^r)\Lambda=M$ (see Lemma \ref{Eq:M}) that $\dim T_r=1$ for $r\in\{\pm1,\pm2, \pm 3\}$.

Let $f$ be the vector of $\Lambda$ defined in \eqref{def_f}.
We regard $f$ as a vector in $(V_\Lambda)_1$ via \eqref{Eq:inj} and define $\sigma_{f}= \exp(-2\pi \sqrt{-1} f_{(0)})$.
Then $\sigma_{f}$ is an automorphism of $V_\Lambda$ of order $7$ because $f\in \frac{1}7 \Lambda\setminus \Lambda$.
By the equality $\tau(f)=f$,  $\sigma_{f}$ commutes with $\tau$, and the automorphism $$g= \sigma_{f}\cdot\tau\in \Aut(V_\Lambda)$$ has order $7$.
Note that $g^r=\sigma_{rf}\cdot\tau^r$.

By Proposition \ref{Prop:twist}, we obtain the irreducible $g^{r}$-twisted $V_\Lambda$-module $V_\Lambda [\tau^{r}]^{(rf)}$ for $r\in\{\pm1,\pm2,\pm3\}$.
For convenience,
we fix a non-zero vector $t_{r}\in T_{r}$.
Then $T_{r}=\C t_{r}$.
By \eqref{Eq:Lh}, we have
\begin{equation}\label{Eq:Lpmf}
L^{(rf)}(0)=L(0)+ rf_{(0)}+\frac{|rf|^2}{2} id.
\end{equation}

\begin{lemma}\label{Lem:twist} For $r\in\{\pm1,\pm2,\pm3\}$, $$\left\{e^{a}\otimes t_{ r}\ \left|\ a\in P_0(\Lambda),\ |a+ r{f} |^2=\frac{2}7\right.\right\}$$ is a basis of $\left(V_\Lambda[\tau^{r}]^{( rf)}\right)_1$.
Moreover,
the dimension of $\left(V_\Lambda[\tau^{ r}]^{( rf)}\right)_1$ is $7$.
\end{lemma}

\begin{proof} Let $w\otimes e^x\otimes t_{ r}\in V_\Lambda[\tau^{ r}]^{( rf)}$ ($w\in M(1)[\tau^r]$, $x\in P_0(\Lambda)$) be a vector whose $L(0)$-weight is $1$.
By \cite[(6.28)]{DL}, it is straightforward to show that the $L(0)$-weight of $t_{r}\in V_\Lambda[\tau^{ r}]$ is $6/7$.
Let $\ell$ be the $L(0)$-weight of $w$ in $M(1)[\tau^r]$, which belongs to $\frac{1}{7}\Z_{\ge0}$.
Then by \eqref{Eq:Lpmf},
the $L(0)$-weight of $w\otimes e^x\otimes t_{ r}$ in the twisted module $V_\Lambda[\tau^{ r}]^{( rf)}$ is
\begin{equation}
\ell+\frac{|x|^2}{2}+\frac{6}7 + r({f}|x)+\frac{|r{f}|^2}2=\ell+\frac{|x+ r{f}|^2}{2}+\frac{6}7,\label{Eq:ell}
\end{equation}
which is equal to $1$ by the assumption.
Hence by $x+rf\neq0$, we have $\ell=0$, and we may assume that $w=1$.
In addition, we obtain
\begin{equation}
|x+r{f}|^2=\frac{2}7.\label{Eq:x1}
\end{equation}
Thus, we have the first assertion. The latter assertion follows from Lemma \ref{Lem:S}. 
\end{proof}

\begin{remark}\label{Rem:pos} For any  $r\in\{\pm1,\pm2,\pm3\}$ and $x\in P_0(\Lambda)$, we have $|x+ r{f}|^2\ge2/7$.
Hence by \eqref{Eq:ell}, the lowest $L(0)$-weight of any irreducible $g^r$-twisted $V_\Lambda$-module is $1$.
\end{remark}

\subsection{$\Z_7$-orbifold construction} 
In this subsection, we discuss the $\Z_7$-orbifold construction from $V_\Lambda$ and $g$. First we consider the following $V_\Lambda^g$-module:
\[
\begin{split}
\tilde{V}_{\Lambda,g}= & V_\Lambda^g\oplus (V_\Lambda[\tau]^{(f)})_\Z\oplus (V_\Lambda[\tau^{2}]^{(2f)})_{\Z}
\oplus (V_\Lambda[\tau^{3}]^{(3f)})_{\Z} \\
& \oplus (V_\Lambda[\tau^{-3}]^{(-3f)})_{\Z}
\oplus (V_\Lambda[\tau^{-2}]^{(-2f)})_{\Z}\oplus (V_\Lambda[\tau^{-1}]^{(-f)})_{\Z},
\end{split}
\]
where $(V_\Lambda[\tau^r]^{(rf)})_\Z$ is the subspace of $V_\Lambda[\tau^r]^{(rf)}$ with integral weight.
By Remark \ref{Rem:pos}, we can apply \cite[Theorem 5.15]{EMS} to our case under the assumption that $V_\Lambda^g$ is regular (see also \cite{CM,Mi4,Mi}).

\begin{proposition}[cf. {\cite[Theorem 5.15]{EMS}}]
The space $\tilde{V}_{\Lambda,g}$ defined as above is a strongly regular, holomorphic VOA of central charge $24$.
\end{proposition}

\begin{proposition}
Let  $\tilde{V}_{\Lambda,g}$ be defined as above. Then $\dim(\tilde{V}_{\Lambda,g})_1= 48$ and the Lie algebra $(\tilde{V}_{\Lambda,g})_1$ has type $A_{6,7}$.
\end{proposition}

\begin{proof}
Recall that $(V_\Lambda)_1$ is an abelian Lie algebra of dimension $24$. 
Viewing $\mathfrak{h}_{(0)}$ as a subspace of $(V_\Lambda)_1$, we know that 
$(V_\Lambda^g)_1=\mathfrak{h}_{(0)}$ is an abelian Lie algebra of dimension $6$.
By Lemma \ref{Lem:twist}, we have $\dim(\tilde{V}_{\Lambda,g})_1= 6+7\times 6=48$. 
By Proposition \ref{Prop:V1}, the Lie algebra $(\tilde{V}_{\Lambda,g})_1$ is semisimple.

Now let $x \in\mathfrak{h}_{(0)}\subset (V_\Lambda^g)_1$. Then
$$x_{(0)}^{( rf)}=x_{(0)}+(x| rf)id$$
on $V_\Lambda[\tau^{ r}]^{(rf)}$. Notice that $x_{(0)}=0$ on $M(1)[\tau^r]$ and on $T_{ r}$ by the explicit description of vertex operators in \cite{Le,DL} (cf.\ \cite{SS}).
Hence for $w\otimes e^{a} \otimes t_{ r}\in (V_\Lambda[\tau^{ r}]^{( rf)})_1$, 
\begin{equation}
x_{(0)}^{( rf)}(w\otimes e^{a} \otimes t_{ r})
= (x| a+  rf ) w\otimes e^{a} \otimes t_{ r}, \label{Eq:actx}
\end{equation}
which shows that $x$ is semisimple in $(\tilde{V}_{\Lambda,g})_1$.
Since $x,a+rf\in\mathfrak{h}_{(0)}$ and $(\cdot|\cdot)$ is non-degenerate on $\mathfrak{h}_{(0)}$, the equation \eqref{Eq:actx} also implies that the centralizer of $\mathfrak{h}_{(0)}$ in $(\tilde{V}_{\Lambda,g})_1$ is again $\mathfrak{h}_{(0)}$. 
Hence $\mathfrak{h}_{(0)}$ is a Cartan subalgebra of $(\tilde{V}_{\Lambda,g})_1$.
Thus $(\tilde{V}_{\Lambda,g})_1$ is a $48$-dimensional semisimple Lie algebra with Lie rank $6$, and the only possibility of its type is $A_6$.
We remark that, up to a scaling, $\{\beta_1, \beta_2, \beta_3, \beta_4, \beta_5, \beta_6\}$ forms a set of simple roots for a root system of type $A_6$ (see Remark \ref{Rem:A6}).

By Lemma \ref{Prop:V1}, we have the ratio $h^\vee/k=1$ and hence the level $k$ is  $7$.
Therefore the Lie algebra $(\tilde{V}_{\Lambda,g})_1$ has type $A_{6,7}$.
\end{proof}

\paragraph{\bf Acknowledgement.} The authors wish to thank the referee for helpful comments.

\end{document}